\documentclass[11pt]{article}

\usepackage[T1]{fontenc}
\usepackage[utf8]{inputenc}

\usepackage{lmodern, microtype}
\usepackage{amssymb, latexsym, textcomp}
\usepackage{amsmath, amsthm}
\usepackage{array, setspace, verbatim}

\usepackage[english]{babel}
\usepackage[english]{varioref}

\usepackage{ifthen}
\usepackage{fancyhdr, fancybox, float}
\usepackage{color, multicol}
\usepackage{graphicx, pst-all, pb-diagram}
\usepackage[pdftex]{hyperref}

\usepackage[a4paper, textwidth=14cm]{geometry}

\newcommand*{\nidt}{\noindent}
\newcommand*{\dst}{\displaystyle}

\newcommand*{\wt}[1]{\widetilde{#1}}

\newcommand*{\espf}{\vspace*{1ex}}

\renewcommand*{\phi}{\varphi}

\newcommand*{\Aa}{\mathcal{A}}
\newcommand*{\C}{\mathbb{C}}
\newcommand*{\Cc}{\mathcal{C}}
\newcommand*{\D}{\mathbb{D}}

\newcommand*{\E}{\mathbb{E}}
\newcommand*{\Ee}{\mathcal{E}}

\renewcommand*{\H}{\mathbb{H}}
\newcommand*{\Ii}{\mathcal{I}}
\newcommand*{\Ll}{\mathcal{L}}
\newcommand*{\R}{\mathbb{R}}
\renewcommand*{\S}{\mathbb{S}}
\newcommand*{\Ss}{\mathcal{S}}

\newcommand*{\Nil}{\text{Nil}}
\newcommand*{\Sol}{\text{Sol}}
\newcommand*{\vol}{\text{vol}}

\newcommand*{\lp}{\left(}
\newcommand*{\rp}{\right)}
\newcommand*{\lc}{\left[}
\newcommand*{\rc}{\right]}
\newcommand*{\lac}{\left\{}
\newcommand*{\rac}{\right\}}
\newcommand*{\lan}{\langle}
\newcommand*{\ran}{\rangle}
\newcommand*{\lb}{\left|}
\newcommand*{\rb}{\right|}
\newcommand*{\rr}{\right.}

\renewcommand*{\a}{\forall}
\newcommand*{\into}{\rightarrow}

\newcommand*{\pint}{\lrcorner}

\newcommand*{\der}[2][]{\frac{\partial#1}{\partial#2}}

\let\div\relax \let\mod\div
\DeclareMathOperator{\div}{div}
\DeclareMathOperator{\mod}{\ mod}
\DeclareMathOperator{\rot}{curl}

\newtheorem{theorem}{Theorem}[section]
\newtheorem{lemma}[theorem]{Lemma}
\newtheorem{corollary}[theorem]{Corollary}
\newtheorem{proposition}[theorem]{Proposition}

\providecommand{\bysame}{\leavevmode\hbox to3em{\hrulefill}\thinspace}

\hypersetup{%
pdfauthor= {Sébastien Cartier},%
pdftitle= {Noether invariants for constant mean curvature surfaces in 3-dimensional homogeneous spaces},%
pdfcreator= {PDFLaTeX},%
pdfproducer= {PDFLaTeX}%
}

\title{Noether invariants for constant mean curvature surfaces in 3-dimensional homogeneous spaces}
\author{Sébastien Cartier}

\begin{document}

\maketitle

\begin{abstract}

We give explicit formulæ for Noether invariants associated to Killing vector fields for the variational problem of minimal and constant mean curvature surfaces in $3$-manifolds. In the case of homogeneous spaces, such invariants are the flux (associated to translations) and the torque (associated to rotations). Then we focus on homogeneous spaces with isometry groups of dimensions $3$ or $4$ and study the behavior of these invariants under the action of isometries. Finally, we give examples of actual computations and of interpretations of these invariants in different situations.

\end{abstract}

\nidt \textit{Mathematics Subject Classification:} \emph{Primary 53C42; Secondary 53A55}.

\section{Introduction}

The differential Noether theorem~\cite{No} describes an isomorphism between the Lie algebra of infinitesimal generators of the \emph{variational symmetries} associated to a variational problem and a space of \emph{conservation laws} for the related Euler-Lagrange equations. In particular, it can be applied to the variational problem of minimal or constant mean curvature (\emph{CMC} for short) surfaces in a homogeneous space using the isometries of the ambient space as symmetries --~for the isometries preserve the mean curvature. In the case of minimal surfaces in the euclidean $3$-space, Noether theorem leads to the notions of flux and torque, which are geometric invariants of the surfaces. And these geometric constrains can be used to find alignment conditions on the catenoidal ends of a surface~\cite{Ro}.

The present paper gives tools to use Noether invariants related to minimal and CMC surfaces in homogeneous manifolds. In Section~\ref{sec:genres}, we give explicit formulæ for Noether forms associated to Killing fields, see Theorem~\ref{thm:noethform} and Proposition~\ref{prop:noethsurf}. In Sections~\ref{sec:noethe} and~\ref{sec:noeths}, we focus on minimal and CMC surfaces in homogeneous spaces $\E^3(\kappa, \tau)$ and $\Sol_3$ respectively, corresponding to isometries of the ambient space. As in the euclidean case, these forms lead to invariants, namely the \emph{flux} and \emph{torque}, related to the geometry of the surface. And in Section~\ref{sec:exples}, we give examples of actual computations of Noether invariants in different situations.

The construction can be written in coordinates using jet bundles~\cite{Ol} or more abstractly using basic tools of contact geometry~\cite{BrGriGro,Ro2}. We choose the second approach, which is coordinate-free and allows us to provide general formulæ.

\section{General results} \label{sec:genres}

\subsection{Contact structure and lagrangians}

The present subsection deals with classical results on contact structures and lagrangians. Details on the notions introduced can be found in~\cite{BrGriGro}.

\medskip

Let $\big{(} M, \lan \cdot, \cdot \ran \big{)}$ be a $3$-dimensional riemannian manifold and consider the following fibration:
\[
FM \stackrel{\pi'}{\longrightarrow} \Cc \stackrel{\pi}{\longrightarrow} M,
\]
where the contact manifold $\Cc$ is the unit fiber bundle $UM$ of $M$ --~or equivalently the Grassmannian of oriented $2$-planes tangent to $M$~-- and $FM$ the orthonormal frame bundle. Since the study is local, we consider a local chart on $M$ with generic point $x$. An element of $\Cc$ is a couple $(x, e_0)$ with $e_0 \in \S^2$ and a point of $FM$ writes $(x, e)$ where $e= (e_0, e_1, e_2)$ is an orthonormal family with respect to $\lan \cdot, \cdot \ran_x$. Finally, the projections $\pi'$ and $\pi$ are respectively:
\[
\pi'(x, e)= (x, e_0) \quad \text{and} \quad \pi(x, e_0)= x.
\]

In the sequel, we work in $FM$ to facilitate computations, but actually the quantities we define are \emph{basic}, i.e. they are liftings of quantities defined on $\Cc$. To ease the understanding, we use the same notation for a quantity and its liftings. Also, we do not distinguish an infinitesimal generator of an action on $M$ from its \emph{extensions} to $\Cc$ or $FM$, i.e. the infinitesimal generators of the natural extension of the action.

\medskip

If $e= (e_0, e_1, e_2)$ is an orthonormal frame on $M$, denote $(\theta^0, \theta^1, \theta^2)$ dual basis composed of $1$-forms and consider elements $(\omega^i_j)_{0 \leq i, j \leq 2}$ of $\Omega^1(FM)$ such that:
\begin{gather*}
d\theta^0= -\omega^0_1 \wedge \theta^1+ \omega^2_0 \wedge \theta^2, \quad d\theta^1= \omega^0_1 \wedge \theta^0- \omega^1_2 \wedge \theta^2, \\
d\theta^2= -\omega^2_0 \wedge \theta^0+ \omega^1_2 \wedge \theta^1 \quad \text{and} \quad \omega^i_j= -\omega^j_i.
\end{gather*}
The \emph{structure forms} $\theta^0$, $\theta^1$, $\theta^2$, $\omega^0_1$, $\omega^1_2$ and $\omega^2_0$ are independent and generate $\Omega^1(FM)$.

\begin{proposition}

Let $\theta^0 \in \Omega^1(\Cc)$ be defined as follows:
\[
\a (x, e_0) \in \Cc, \ \a (u, \xi) \in T_{(x, e_0)} \Cc, \ \theta^0_{(x, e_0)}(u, \xi)= \lan e_0, u \ran_x.
\]
If $I$ is the line subfiber bundle of $T^* \Cc$ generated by $\theta^0$, then $(\Cc, I)$ is a contact structure and in the sequel we call $\theta^0$ the \emph{contact form}.

\end{proposition}

The \emph{contact ideal} $\Ii \subset \Omega^*(\Cc)$ is the ideal --~with respect to the exterior product~-- generated by $\lac \theta^0, d\theta^0 \rac$. Lifting $e_0$ to an element $(e_0, e_1, e_2)$ of $FM$ with dual basis $(\theta^0, \theta^1, \theta^2)$, the $1$-form $\theta^0$ on $FM$ coincides with the lifting of the contact form, which is why they are denoted the same.

If $f: \Sigma \into M$ is an immersion of a simply connected surface $\Sigma$, there exists a \emph{legendrian lift} $N: \Sigma \into \Cc$ of $f$ to $\Cc$, which means that $N$ verifies $N^* \theta^0= 0$ and $f= \pi \circ N$. Note that, by construction of $\theta^0$, the lift $N$ is unique up to sign and it is a normal vector to $f$. Moreover, the condition $N^* \theta^0= 0$ implies $N^* d\theta^0= 0$, and thus $N^* \Ii= \lac 0 \rac$.

\medskip

The study is local, so we can assume $\Sigma$ is compact, eventually with boundary. Consider the functional $\Aa$ such that:
\[
\Aa(\Sigma)= \int_{\Sigma} N^* 	\Lambda_0 \quad \text{with} \quad \Lambda_0= e_0 \pint \vol_M,
\]
where $\vol_M$ is the volume form on $M$. We call $\Lambda_0$ the \emph{lagrangian} of the functional. Actually, $\Aa$ is the area functional, since an expression of $\Lambda_0$ in $FM$ is $\Lambda_0= \theta^1 \wedge \theta^2$, the volume form being $\vol_M= \theta^0 \wedge \theta^1 \wedge \theta^2$. A classical result on the area functional is the following:

\begin{proposition}

Let $f: \Sigma \into M$ be an immersion with legendrian lift $N: \Sigma \into \Cc$. Then $f$ is a critical point of the functional $\Aa$ if and only if the \emph{Euler-Lagrange condition} $N^* \Psi_0$ is satisfied, where:
\[
\Psi_0= -\omega^2_0 \wedge \theta^1- \omega^0_1 \wedge \theta^2
\]
is the \emph{Euler-Lagrange operator}. Moreover, if $f$ is a critical point of $\Aa$, then the Euler-Lagrange condition means that it is a minimal immersion.

\end{proposition}

In the following, fix $H$ a constant, eventually zero. The variational characterization of CMC-$H$ immersions in $M$ can be deduced from the previous result on minimal immersions by adding a Lagrange multiplier to grasp the volume constraint when $H \neq 0$. Remark first that:

\begin{lemma}

In any riemannian manifold $(M, g)$ of finite dimension, there exists locally a vector field $\Xi \in \mathfrak{X}(M)$, which we call a \emph{volume field}, such that:
\[
\div_M \Xi= 1.
\]

\end{lemma}

We use such a vector field to write the lagrangian involved in the variational characterization of CMC-$H$ immersion:

\begin{proposition}

Let $\Xi$ be a volume field on $M$. The lagrangian $\Lambda$ defined on $\Cc$ by:
\begin{equation} \label{eq:lagrcmc}
\Lambda= \Lambda_0+ 2H\Lambda' \quad \text{with} \quad \Lambda'= \Xi \pint \vol_M,
\end{equation}
is associated to the variational problem of CMC-$H$ immersions in $M$. In other words, an immersion $f: \Sigma \into M$ with legendrian lift $N: \Sigma \into \Cc$ is a critical point of the functional associated to $\Lambda$ if and only if it is a CMC-$H$ immersion. Furthermore, the Euler-Lagrange operator writes $\Psi= \Psi_0+ 2H\Lambda_0$ and the Euler-Lagrange equation is $N^* \Psi= 0$.

\end{proposition}

We define the \emph{Euler-Lagrange system} as the differential ideal $\Ee \subset \Omega^*(\Cc)$ generated by $\lac \theta^0, d\theta^0, \Psi \rac$. Hence, to determine a Noether form related to a minimal or CMC immersion $f: \Sigma \into M$, we only need to compute a class of forms on $\Cc$ modulo the ideal $\Ee$ and pull it back in $\Omega^1(\Sigma)$.

\subsection{Symmetries and Noether forms}

We call a \emph{divergence symmetry} of the variational problem with lagrangian $\Lambda$ defined by~\eqref{eq:lagrcmc}, any element $S \in \mathfrak{X}(\Cc)$ for which there exists a class $\Phi_S \in H^1(\Cc)$ such that $\Ll_S \Lambda \equiv d\phi \mod \Ee$, for an arbitrary $\phi \in \Phi_S$. Noether theorem states then:

\begin{theorem}[Noether, 1918~\cite{No}]

Any divergence symmetry $S$ is in one-to-one correspondence with a class of $1$-forms $\mu_S \in H^1(\Cc)/ \Ee$ defined by:
\[
\mu_S= S\pint \Lambda- \phi \quad \text{in } H^1(\Cc)/ \Ee \quad \text{with} \quad \phi \in \Phi_S.
\]
Moreover, if $N: \Sigma \into \Cc$ is the legendrian lift of a critical point of $\Lambda$ --~i.e. a CMC-$H$ immersion~--, then the pull back $N^* \mu_S$ is a closed form on $\Sigma$ and the quantity:
\[
\sigma_S(c)= \int_c N^* \mu_S
\]
is the \emph{Noether invariant} or \emph{conserved quantity} associated to $S$ along the cycle $c \in H_1(\Sigma)$.

\end{theorem}

Note that the closure condition for $N^* \mu_S$ is the \emph{conservation law} mentioned in the introduction. The expression of the Noether form can be made completely explicit in the case of Killing fields:

\begin{theorem} \label{thm:noethform}

Let $S \in \mathfrak{X}(M)$ be a Killing field. Then the extension of $S$ to $\Cc$ is a divergence symmetry. Furthermore, if $F \in \mathfrak{X}(\Cc)$ is the extension of a \emph{potential vector} of $S$ --~i.e. a field $F \in \mathfrak{X}(M)$ such that $\rot_M F= S$~--, then the corresponding Noether form $\mu_S$ writes:
\[
\mu_S= S\pint \Lambda_0- 2HF^{\flat} \quad \text{in } H^1(\Cc)/ \Ee,
\]
and this expression does not depend on the choice of the potential vector.

\end{theorem}

\begin{proof}
Since $S$ is Killing field on $M$, we have $\div_M S= 0$. Thus, there exists a field $F \in \mathfrak{X}(M)$ such that $S= \rot_M F$. Moreover, we have $\Ll_S \Lambda \equiv 2H\Ll_S \Lambda' \mod \Ee$, with by definition $\Ll_S \Lambda'= d(S\pint \Lambda')+ S\pint d\Lambda'$. If $*$ denotes the Hodge operator and $\cdot^{\flat}, \cdot^{\sharp}$ the musical isomorphisms, we know that:
\[
S= \rot_M F= (*dF^{\flat})^{\sharp} \quad \text{and} \quad S\pint d\Lambda'= S\pint \vol_M= *S^{\flat}= *^2 dF^{\flat}= dF^{\flat}.
\]
Hence:
\[
\Ll_S \Lambda \equiv 2Hd(S\pint \Lambda'+ F^{\flat}) \mod \Ee,
\]
and $S$ is indeed a divergence symmetry. The Noether form associated to $S$ is:
\[
\mu_S= S\pint \Lambda- 2H\lp S\pint \Lambda'+ F^{\flat} \rp= S\pint \Lambda_0- 2HF^{\flat} \quad \text{in } H^1(\Cc)/ \Ee,
\]
and this expression does not depend on the choice of the potential vector $F$ since $d(F^{\flat}- \wt{F}^{\flat})= 0$ for any other choice $\wt{F}$ of a potential vector of $S$.
\end{proof}

\begin{corollary}

If $M$ is a homogeneous space, the extensions of infinitesimal generators of $1$-parameter families of isometries are divergence symmetries.

\end{corollary}

Let $f: \Sigma \into M$ be an (oriented) CMC-$H$ immersion. We choose its legendrian lift $N: \Sigma \into \Cc$ so that it coincides with the unit normal to $f$. Since $df= e_1 \theta^1+ e_2 \theta^2$ and $*df= -e_2 \theta^1+ e_1 \theta^2$, we have:

\begin{proposition} \label{prop:noethsurf}

The pullback $N^* \mu_S$ is well defined in $H^1(\Sigma)$ and writes:
\[
N^* \mu_S= N^* \mu_S^0- 2HN^* \mu_S' \quad \text{with} \quad \mu_S^0= \lan S, *df \ran \quad \text{and} \quad \mu_S'= \lan F, df \ran.
\]
We call $\mu_S^0$ the \emph{minimal part} of the Noether form and $\mu_S'$ its \emph{CMC part}.

\end{proposition}

In the cases of homogeneous spaces $\E^3(\kappa, \tau)$ and $\Sol_3$ (Sections~\ref{sec:noethe} and~\ref{sec:noeths} respectively), we denote $\sigma_i(\cdot)$, with $i= 1, 2, 3, R$, the Noether invariants corresponding to isometries ($1, 2, 3$ for translations and $R$ for the rotation when it exists). The \emph{flux} through a cycle $c \in H_1(\Sigma)$ is the vector $\sigma(c)= \big{(} \sigma_1(c), \sigma_2(c), \sigma_3(c) \big{)}$ and the \emph{torque} is the number $\sigma_R(c)$ when it exists.

If $\Ss(t)$ is a $1$-parameter family of isometries with infinitesimal generator $S$, we know from Proposition~\ref{prop:noethsurf} that determining the corresponding Noether form $\mu_S$ restricted to $f$ means actually computing $S$ and a potential vector $F$.

It is also interesting to study the relations between Noether invariant of congruent immersions --~i.e. immersions deduced from on another by the action of an isometry. Namely, considering $\Ss'(t)$ a(nother) $1$-parameter family of isometries, we want to compare the form $\mu_S$ in restriction to an immersion $f$ and the corresponding form denoted $\mu_S(\Ss'(t))$ in restriction to the immersion $\Ss'(t) \circ f$. Remark that:
\[
\mu_S(\Ss'(t))= \lan d\Ss'(t)^{-1} \cdot S(\Ss(t)), *df \ran- 2H\lan d\Ss'(t)^{-1} \cdot F(\Ss(t)), df \ran.
\]

\section{Noether forms in $\E^3(\kappa, \tau)$} \label{sec:noethe}

The spaces $\E^3(\kappa, \tau)$ are simply connected $3$-dimensional homogeneous spaces. They are characterized by real parameters $\kappa$ and $\tau$ such that $\kappa- 4\tau^2 \neq 0$. The model considered is $\Omega_{\kappa} \times \R \subset \R^3$, with generic coordinates $(w= x_1+ ix_2, x_3)$ and:
\[
\Omega_{\kappa}= \lac \begin{array}{ll}
\C                    & \text{if } \kappa \geq 0 \\
\D(2|\kappa|^{-1/ 2}) & \text{if } \kappa< 0
\end{array} \rr,
\]
endowed with the metric:
\[
ds^2= \lambda^2 |dw|^2+ \big{(} \tau \lambda (x_2 dx_1- x_1 dx_2)+ dx_3 \big{)}^2 \quad \text{with} \quad \lambda= \frac{1}{1+ \kappa' |w|^2} \quad \text{and} \quad \kappa'= \frac{\kappa}{4}.
\]

These spaces are riemannian fibrations of the \emph{base} $\Omega_{\kappa}$ for the natural projection $\E^3(\kappa, \tau) \into \Omega_{\kappa}$ on the first two coordinates. The parameter $\kappa$ can be interpreted as the curvature of the base and $\tau$ as the one of the fibration. Thus, the space $\E^3(\kappa, \tau)$ has the geometry of Berger spheres if $\kappa> 0$, the one of the Heisenberg group if $\kappa= 0$ and the geometry of the universal cover of $\mathrm{PSL}_2(\R)$ if $\kappa< 0$ --~in the latter however, the model used corresponds to the universal cover of $\mathrm{PSL}_2(\R)$ minus a fiber. In Section~\ref{sec:exples}, we focus on the Heisenberg group $\Nil_3= \E^3(0, 1/ 2)$ and the product space $\H^2 \times \R= \E^3(-1, 0)$.

We consider the orthonormal frame $(E_1, E_2, E_3)$ defined when $\tau \neq 0$ by:
\begin{gather*}
E_1= \frac{1}{\lambda} \lp \cos(\sigma x_3) \der{x_1}+ \sin(\sigma x_3) \der{x_2} \rp+ \tau \big{(} x_1 \sin(\sigma x_3)- x_2 \cos(\sigma x_3) \big{)} \der{x_3}, \\
E_2= \frac{1}{\lambda} \lp -\sin(\sigma x_3) \der{x_1}+ \cos(\sigma x_3) \der{x_2} \rp+ \tau \big{(} x_1 \cos(\sigma x_3)+ x_2 \sin(\sigma x_3) \big{)} \der{x_3} \\
\text{and} \quad E_3= \der{x_3} \quad \text{with} \quad \sigma= \frac{\kappa}{2\tau},
\end{gather*}
and if $\tau= 0$:
\[
E_1= \frac{1}{\lambda} \der{x_1}, \quad E_2= \frac{1}{\lambda} \der{x_2} \quad \text{and} \quad E_3= \der{x_3}.
\]
The space spanned by $E_1, E_2$ is said to be \emph{horizontal} and the vector field $E_3$ is a Killing field.

\subsection{Isometries and curl operator}

A natural volume field is $\Xi= x_3 E_3$. Note that in the case of Berger spheres ($\kappa> 0$ and $\tau \neq 0$), this field is not globally defined.

The isometry group of $\E^3(\kappa, \tau)$ is $4$-dimensional, generated by three $1$-parameter families of translations et one of rotations:
\begin{gather*}
\Ss_1(t)(w, x_3)= \lp \frac{t+ w}{1- \kappa' tw}, x_3+ \frac{4}{\sigma} \arctan \lp \frac{\kappa' tx_2}{1- \kappa' tx_1+ |1- \kappa' tw|} \rp \rp, \\
\Ss_2(t)(w, x_3)= \lp \frac{it+ w}{1+ i\kappa' tw}, x_3- \frac{4}{\sigma} \arctan \lp \frac{\kappa' tx_1}{1- \kappa' tx_2+ |1+ i\kappa' tw|} \rp \rp, \\
\Ss_3(t)(w, x_3)= (w, x_3+ t) \quad \text{and} \quad \Ss_R(t)(w, x_3)= \lp we^{it}, x_3 \rp \quad \text{with} \quad \sigma= \frac{\kappa}{2\tau}.
\end{gather*}
Infinitesimal generators of these families are respectively:
\begin{gather*}
S_1= \lp 1+ \kappa' (x_1^2- x_2^2) \rp \der{x_1}+ 2\kappa' x_1 x_2 \der{x_2}+ \tau x_2 \der{x_3}, \\
S_2= 2\kappa' x_1 x_2 \der{x_1}+ \lp 1- \kappa' (x_1^2- x_2^2) \rp \der{x_2}- \tau x_1 \der{x_3}, \\
S_3= \der{x_3} \quad \text{and} \quad S_R= -x_2 \der{x_1}+ x_1 \der{x_2}.
\end{gather*}

Let $X \in \mathfrak{X} \big{(} \E^3(\kappa, \tau) \big{)}$ be written as $X= X^1E_1+ X^2E_2+ X^3E_3$. The expression of $\rot X$ depends on $\tau$. If $\tau \neq 0$:
\begin{multline*}
\rot X= \lp dX^3(E_2)- dX^2(E_3)- \sigma X^1 \rp E_1+ \lp dX^1(E_3)- dX^3(E_1)- \sigma X^2 \rp E_2 \\
+\lp dX^2(E_1)- dX^1(E_2)- 2\tau X^3 \rp E_3,
\end{multline*}
and if $\tau= 0$ we have:
\begin{multline*}
\rot X= \lp dX^3(E_2)- dX^2(E_3) \rp E_1+ \lp dX^1(E_3)- dX^3(E_1) \rp E_2 \\
+\lp dX^2(E_1)- dX^1(E_2)+ 2\kappa' (x_2 X^1- x_1 X^2) \rp E_3.
\end{multline*}
For the horizontal translations and the rotation, the potential vectors are, if $\kappa \neq 0$:
\[
F_1= \frac{1}{\sigma} S_1^h+ \lambda x_2 E_3, \quad F_2= \frac{1}{\sigma} S_2^h- \lambda x_1 E_3 \quad \text{and} \quad F_R= \frac{1}{\sigma} S_R^h+ \frac{\lambda}{2\kappa'} E_3,
\]
where $\cdot^h$ denotes the horizontal part, and if $\kappa= 0$:
\[
F_1= (\tau x_1 x_2- x_3) E_2, \quad F_2= (\tau x_1 x_2+ x_3) E_1 \quad \text{and} \quad F_R= x_1 x_3 E_1+ x_2 x_3 E_2.
\]
The case of the vertical translation is discriminated by $\tau$:
\[
F_3= \lac \begin{array}{ll}
\dst -\frac{1}{2\tau} E_3                  & \text{if } \tau \neq 0 \espf \\
\dst -\frac{x_2}{2} E_1+ \frac{x_1}{2} E_2 & \text{if } \tau= 0
\end{array} \rr.
\]

\subsection{Evolution under the action of isometries} \label{subsec:evolisome}

\subsubsection{If $\tau \neq 0$}

We have the following behavior for the Noether forms:
\begin{gather*}
\mu_1(\Ss_1(t))= \mu_1(\Ss_3(t))= \mu_1, \quad \mu_1(\Ss_2(t))= \frac{1- \kappa' t^2}{1+ \kappa' t^2} \mu_1+ \frac{2t}{1+ \kappa' t^2} \lp 2\kappa' \mu_R+ \tau \mu_3 \rp \\
\text{and} \quad \mu_1(\Ss_R(t))= \cos t \mu_1- \sin t \mu_2, \\
\mu_2(\Ss_1(t))= \frac{1- \kappa' t^2}{1+ \kappa' t^2} \mu_2- \frac{2t}{1+ \kappa' t^2} \lp 2\kappa' \mu_R+ \tau \mu_3 \rp, \quad \mu_2(\Ss_2(t))= \mu_2(\Ss_3(t))= \mu_2, \\
\text{and} \quad \mu_2(\Ss_R(t))= \cos t \mu_2+ \sin t \mu_1,
\end{gather*} \begin{gather*}
\mu_3(\Ss_1(t))= \mu_3(\Ss_2(t))= \mu_3(\Ss_3(t))= \mu_3(\Ss_R(t))= \mu_3, \\
\mu_R(\Ss_1(t))= \frac{1- \kappa' t^2}{1+ \kappa' t^2} \mu_R+ \frac{t}{1+ \kappa' t^2} \lp \mu_2- \tau t \mu_3 \rp, \\
\mu_R(\Ss_2(t))= \frac{1- \kappa' t^2}{1+ \kappa' t^2} \mu_R- \frac{t}{1+ \kappa' t^2} \lp \mu_1+ \tau t \mu_3 \rp \\
\text{and} \quad \mu_R(\Ss_3(t))= \mu_R(\Ss_R(t))= \mu_R.
\end{gather*}

We deduce the values of Noether forms depending on the symmetries of the surface:

\begin{proposition}

Let $f: \Sigma \into \E^3(\kappa, \tau)$, $\tau \neq 0$, be a minimal or CMC immersion. We have the following assertions:

\begin{itemize}

	\item[(i)] Suppose $f$ is invariant under the action of a translation $\Ss_1(t)$ (resp. $\Ss_2(t)$). Then $\mu_2= \mu_3= 0$ (resp. $\mu_1= \mu_3= 0$) if $\kappa= 0$, and $\mu_2= \kappa \mu_R+ 2\tau \mu_3= 0$ (resp. $\mu_1= \kappa \mu_R+ 2\tau \mu_3= 0$) if $\kappa \neq 0$.

	\item[(ii)] Suppose $f$ is invariant for a rotation $\Ss_R(t)$. If $c$ is a cycle homologous to its image $\Ss_R(t) \cdot c$, then $\mu_1= \mu_2= 0$.

\end{itemize}

\end{proposition}

\subsubsection{If $\tau= 0$}

Noether forms associated to horizontal translations and the rotation evolve the same:
\begin{gather*}
\mu_1(\Ss_1(t))= \mu_1(\Ss_3(t))= \mu_1, \quad \mu_1(\Ss_2(t))= \frac{1- \kappa' t^2}{1+ \kappa' t^2} \mu_1+ \frac{4\kappa' t}{1+ \kappa' t^2} \mu_R, \\
\text{and} \quad \mu_1(\Ss_R(t))= \cos t \mu_1- \sin t \mu_2, \\
\mu_2(\Ss_1(t))= \frac{1- \kappa' t^2}{1+ \kappa' t^2} \mu_2- \frac{4\kappa' t}{1+ \kappa' t^2} \mu_R, \quad \mu_2(\Ss_2(t))= \mu_2(\Ss_3(t))= \mu_2, \\
\text{and} \quad \mu_2(\Ss_R(t))= \cos t \mu_2+ \sin t \mu_1, \\
\mu_R(\Ss_1(t))= \frac{1- \kappa' t^2}{1+ \kappa' t^2} \mu_R+ \frac{t}{1+ \kappa' t^2} \mu_2, \quad
\mu_R(\Ss_2(t))= \frac{1- \kappa' t^2}{1+ \kappa' t^2} \mu_R- \frac{t}{1+ \kappa' t^2} \mu_1 \\
\text{and} \quad \mu_R(\Ss_3(t))= \mu_R(\Ss_R(t))= \mu_R.
\end{gather*}
However, for the form corresponding to $S_3$, the minimal part remains the same and the CMC part verifies:
\begin{gather*}
\mu_3'(\Ss_1(t))= \frac{1}{|1- \kappa' tw|^2} \mu_3'+ \frac{t}{2\lambda |1- \kappa' tw|^2} \lan S_2, df \ran, \\
\mu_3'(\Ss_2(t))= \frac{1}{|1+ i\kappa' tw|^2} \mu_3'- \frac{t}{2\lambda |1+ i\kappa' tw|^2} \lan S_1, df \ran, \\
\mu_3'(\Ss_3(t))= \mu_3' \quad \text{and} \quad \mu_3'(\Ss_R(t))= \mu_3'.
\end{gather*}

\begin{proposition}

Let $f: \Sigma \into \E^3(\kappa, 0)$ be a minimal or CMC immersion. We have the following assertions:

\begin{itemize}

	\item[(i)] Suppose $f$ is invariant under the action of a translation $\Ss_1(t)$ (resp. $\Ss_2(t)$). Then $\mu_2= \mu_R= 0$ (resp. $\mu_1= \mu_R= 0$).

	\item[(ii)] Suppose $f$ is invariant for a rotation $\Ss_R(t)$. If $c$ is a cycle homologous to $\Ss_R(t) \cdot c$, then $\mu_1= \mu_2= 0$.

\end{itemize}

\end{proposition}

\section{Noether forms in $\Sol_3$} \label{sec:noeths}

The space $\Sol_3$ is the $3$-dimensional Lie group:
\[
\Sol_3= \lp \lac (x_1, x_2, x_3) \in \R^3 \rac, \ ds^2= e^{2x_3} dx_1^2+ e^{-2x_3} dx_2^2+ dx_3^2 \rp.
\]
We can also define a canonical frame $(E_1, E_2, E_3)$ on $\Sol_3$, given by:
\[
E_1= e^{-x_3} \der{x_1}, \quad E_2= e^{x_3} \der{x_2} \quad \text{and} \quad E_3= \der{x_3}.
\]

\subsection{Isometries and curl operator}

As in the $\E^3(\kappa, \tau)$, a natural volume field is $\Xi= x_3 E_3$.

\medskip

The isometry group of $\Sol_3$ is of dimension $3$, generated by the following three $1$-parameter families of translations:
\begin{gather*}
\Ss_1(t)(x_1, x_2, x_3)= (x_1+ t, x_2, x_3), \quad \Ss_2(t)(x_1, x_2, x_3)= (x_1, x_2+ t, x_3) \\
\text{and} \quad \Ss_3(t)(x_1, x_2, x_3)= (e^{-t} x_1, e^t x_2, x_3+ t).
\end{gather*}
The infinitesimal generators are respectively:
\[
S_1= \der{x_1}, \quad S_2= \der{x_2} \quad \text{and} \quad S_3= -x_1 \der{x_1}+ x_2 \der{x_2}+ \der{x_3}.
\]

\medskip

Consider $X \in \mathfrak{X}(\Sol_3)$ written $X= X^1E_1+ X^2E_2+ X^3E_3$ in the canonical frame. The curl of $X$ is:
\begin{multline*}
\rot X= \lp dX^3(E_2)- dX^2(E_3)+ X^2 \rp E_1+ \lp dX^1(E_3)- dX^3(E_1)+ X^1 \rp E_2 \\
+\lp dX^2(E_1)- dX^1(E_2) \rp E_3.
\end{multline*}
We deduce expressions of potential vectors:
\[
F_1= x_2 E_3, \quad F_2= -x_1 E_3 \quad \text{and} \quad F_3= -\frac{x_2 e^{-x_3}}{2} E_1+ \frac{x_1 e^{x_3}}{2} E_2- x_1 x_2 E_3.
\]

\subsection{Evolution under the action of isometries}

The expressions of the Noether forms are simpler in the case of $\Sol_3$ than in the $\E^3(\kappa, \tau)$. We have directly:
\begin{gather*}
\mu_1(\Ss_1(t))= \mu_1(\Ss_2(t))= \mu_1 \quad \text{and} \quad \mu_1(\Ss_3(t))= e^t \mu_1, \\
\mu_2(\Ss_1(t))= \mu_2(\Ss_2(t))= \mu_2 \quad \text{and} \quad \mu_2(\Ss_3(t))= e^{-t} \mu_2, \\
\mu_3(\Ss_1(t))= \mu_3- t\mu_1, \quad \mu_3(\Ss_2(t))= \mu_3+ t\mu_2 \quad \text{and} \quad \mu_3(\Ss_3(t))= \mu_3.
\end{gather*}

\begin{proposition}

Let $f: \Sigma \into \Sol_3$ be a minimal or CMC immersion. If $f$ is invariant under the action of a horizontal translation $\Ss_1(t)$ (resp. $\Ss_2(t)$), then $\mu_1= 0$ (resp. $\mu_2= 0$). And if $f$ is invariant for a vertical translation $\Ss_3(t)$, then $\mu_1= \mu_2= 0$.

\end{proposition}

\section{Examples} \label{sec:exples}

\subsection{Vertical catenoids in $\Nil_3$}

In the Heisenberg group, using notations of P.~Bérard and M.~P.~Calvacante in~\cite{BeCa}, the vertical catenoids come as a $1$-parameter family $(\Cc_a)$ of rotationally invariant minimal, where $a$ is a positive parameter. A catenoid $\Cc_a$, for some $a> 0$, can be parametrized as:
\[
X_a: (t, \theta) \in \R \times [0, 2\pi] \mapsto \big{(} f(a, t)\cos \theta, f(a, t)\sin \theta, t \big{)} \in \Nil_3,
\]
where $f(a, \cdot)$ is a positive function which is a global solution of the Cauchy problem:
\[
f(f^2+ 4)f_{tt}= 4(1+ f_t^2), \quad f(0)= a \quad \text{and} \quad f_t(0)= 0.
\]
The parameter $a$ is indeed the \emph{size of the neck}. Consider for any fixed $t \in \R$, the closed curve $\Cc_a \cap \lac x_3= t \rac$, parametrized by:
\[
\theta \in [0, 2\pi] \mapsto \big{(} f(a, t)\cos \theta, f(a, t)\sin \theta, t \big{)} \in \Nil_3
\]
An orthonormal basis $(e_1, e_2)$ of the tangent space to $\Cc_a$ is:
\begin{gather*}
e_1= \frac{2}{\sqrt{4+ f^2}} \lp -\sin \theta E_1+ \cos \theta E_2- \frac{f}{2} E_3 \rp \\
\text{and} \quad e_2= -\frac{\sqrt{4+ f^2}}{\sqrt{4+ f_t^2(4+ f^2)}} \bigg{[} \lp f_t\cos \theta- \frac{2f}{4+ f^2} \sin \theta \rp E_1 \hspace{8em} \\
\hspace{15em}+\lp f_t\sin \theta+ \frac{2f}{4+ f^2} \cos \theta \rp E_2+ \frac{4}{4+ f^2} E_3 \bigg{]},
\end{gather*}
with $e_1$ tangent to the curve $\Cc_a \cap \lac x_3= t \rac$.

If $S$ is the infinitesimal generator of a $1$-parameter family of isometries, we have:
\[
\sigma_S= -\frac{f}{2} \sqrt{4+ f^2} \int_0^{2\pi} \lan S, e_2 \ran d\theta,
\]
and, along the curve, the infinitesimal generators $S_1, S_2, S_3, S_R$ are:
\begin{gather*}
S_1= E_1+ f\sin \theta E_3, \quad S_2= E_2- f\cos \theta E_3, \quad S_3= E_3 \\
\text{and} \quad S_R= f(-\sin \theta E_1+ \cos \theta E_2)- \frac{f^2}{2} E_3.
\end{gather*}
We obtain $\sigma_1= \sigma_2= \sigma_R= 0$ and:
\[
\sigma_3= 2\pi \frac{2f}{\sqrt{4+ f_t^2(4+ f^2)}}.
\]
This expression of $\sigma_3$ is given in~\cite[Proposition~2.2]{BeCa} up to the $2\pi$ factor. Moreover, as $\sigma_3$ is independent of $t$, we can make $t= 0$ in its expression to get $\sigma_3= 2\pi a$. Hence, it appears that the vertical flux controls the size of the neck.

\subsection{Horizontal catenoids in $\Nil_3$}

We follow the notations of B.~Daniel and L.~Hauswirth in~\cite{DaHa}. Consider the coordinates $(y_1, y_2, y_3)$ on $\Nil_3$ defined by:
\[
y_1= x_1, \quad y_2= x_2 \quad \text{and} \quad y_3= x_3+ \frac{x_1 x_2}{2}.
\]
The metric is $dy_1^2+ dy_2^2+(y_1 dy_2- dy_3)^2$ and the change of basis on the tangent space writes:
\[
\der{y_1}= \der{x_1}- \frac{x_2}{2} \der{x_3}, \quad \der{y_2}= \der{x_2}- \frac{x_1}{2} \der{x_3} \quad \text{and} \quad \der{y_3}= \der{x_3}.
\]

In these coordinates, the immersion $f_{\alpha}= (F_1, F_2, h): \C \into \Nil_3$ describing thee catenoid $\Cc_{\alpha}$ of parameter $\alpha> 0$ is:
\begin{gather*}
F_1(u, v)= \frac{G'(u)}{\alpha} \cos \phi(u) \sinh A(u, v)- \frac{C}{\alpha} \sin \phi(u) \cosh A(u, v), \\
F_2(u, v)= \frac{C}{\alpha} A(u, v)- \frac{C}{\alpha} \beta(u)- G(u)
\end{gather*} \vspace{-\belowdisplayskip} \vspace{-\abovedisplayskip} \begin{multline*}
\text{and} \quad h(u, v)= \frac{C}{\alpha} \lp \frac{G'(u)}{\alpha}- 1 \rp \cos \phi(u) \cosh A(u, v) \\
-\frac{1}{\alpha} \lp \frac{C^2}{\alpha}+ G'(u) \rp \sin \phi(u) \sinh A(u, v),
\end{multline*}
with $C, \phi, \beta, A, G$ defined as in~\cite{DaHa}, i.e. $C= \sin(2\theta)/ (2\alpha)$, $\phi$ is solution of the ODE:
\[
\phi'^2= \alpha^2+ \cos(2\theta) \cos^2 \phi- C^2 \cos^4 \phi,
\]
$\beta$ and $G$ are respectively defined by:
\[
\lac \begin{array}{l}
\beta'= C\cos^2 \phi \\
\beta(0)= 0
\end{array} \rr \quad \text{and} \quad \lac \begin{array}{l}
\dst G'= \frac{C^2 \cos^2 \phi- \cos(2\theta)}{\alpha- \phi'} \\
G(0)= 0
\end{array} \rr,
\]
$A= \alpha v+ \beta(u)$ and the parameter $\theta$ is chosen as solution of the equation:
\[
\int_{-1}^1 \frac{2\alpha C^2 t^2- \alpha \cos(2\theta)+ C^2 t^2 \sqrt{P(t)}}{\sqrt{(1- t^2)P(t)} \lp \alpha+ \sqrt{P(t)} \rp} dt= 0 \quad \text{with} \quad P(t)= \alpha^2+ \cos(2\theta) t^2- C^2 t^4.
\]

Consider the closed convex curve $\Cc_{\alpha} \cap \lac y_2= t \rac$ which is of period $2U$, naturally parametrized by:
\[
u \in [0, 2U] \mapsto \big{(} F_1(u, v), t, h(u, v) \big{)} \in \Nil_3,
\]
with the following expression of $A$ on the curve:
\[
A(u, v)= \frac{\alpha}{C} t+ \beta(u)+ \frac{\alpha}{C} G(u).
\]
An orthonormal basis $(e_1, e_2)$ of the tangent space to $\Cc_{\alpha}$ is:
\begin{gather*}
e_1= \cos \phi E_1- \sin \phi E_3 \\
\text{and} \quad e_2= \frac{1}{\cosh A} \big{(} -(\sin \phi \sinh A)E_1+ E_2- (\cos \phi \sinh A)E_3 \big{)},
\end{gather*}
with $e_1$ tangent to the curve $\Cc_{\alpha} \cap \lac y_2= t \rac$.

If $S$ is the infinitesimal generator of a $1$-parameter family of isometries, we have:
\[
\sigma_S= -\frac{1}{C} \int_0^{2U} (C^2+ G'^2) \lan S, \cosh A \, e_2 \ran du,
\]
and the infinitesimal generators $S_1, S_2, S_3, S_R$ write as follow along the curve:
\begin{gather*}
S_1= E_1+ tE_3, \quad S_2= E_2- F_1 E_3, \quad S_3= E_3 \\
\text{and} \quad S_R= -tE_1+ F_1 E_2- \frac{F_1^2+ t^2}{2} E_3.
\end{gather*}
We obtain:
\begin{gather*}
\sigma_1= \frac{1}{C} \int_0^{2U} (C^2+ G'^2)(\sin \phi+ t\cos \phi) \sinh A\ du, \\
\sigma_2= -\frac{1}{C} \int_0^{2U} (C^2+ G'^2)(1+ F_1\cos \phi~\sinh A) du, \\
\sigma_3= \frac{1}{C} \int_0^{2U} (C^2+ G'^2)\cos \phi~\sinh A\ du \\
\text{and} \quad \sigma_R= -\frac{1}{C} \int_0^{2U} (C^2+ G'^2) \lc t\sin \phi~\sinh A+ F_1+ \frac{F_1^2+ t^2}{2} \cos \phi~\sinh A \rc du.
\end{gather*}
These quantities are homological invariant and thus independent of the parameter $t$. It implies:
\[
\sigma_1= \sigma_3= \sigma_R= 0 \quad \text{and} \quad \sigma_2= \frac{1}{2\alpha C} \int_0^{2U} (C^2+ G'^2)(G' \cos^2 \phi- 2\alpha)du.
\]
Moreover, using relations between $G$ and $\phi$~\cite[page~14]{DaHa}, we get:
\[
\sigma_2= \frac{\cos(2\wt{\theta}_{\alpha})}{\alpha C} G(U)- 2CU.
\]

\subsection{CMC-$1/ 2$ vertical ends in $\H^2 \times \R$}

In a recent paper~\cite{CaHa}, the author and L.~Hauswirth have constructed entire graphs and annuli of constant mean curvature $1/ 2$ in $\H^2 \times \R$ with prescribed asymptotic behavior seen as deformations of rotational examples, and it appears that the existence conditions are flux conditions. The constructed surfaces have vertical ends, which means the ends are properly immersed topological annuli with asymptotically horizontal normal vector.

Following the notations in~\cite{CaHa}, the surfaces are parametrized in the Poincaré disk model of $\H^2 \times \R$ by:
\[
X_{\beta}^{\eta}: re^{i\theta} \in \Omega \mapsto \lp \frac{2r}{1+ r^2} e^{i\theta}, e^{\eta(r, \theta)} h_{\beta}(r) \rp \in \H^2 \times \R,
\]
where $\beta> 0$ is a real parameter, $(r, \theta)$ are the polar coordinates on $\R^2$, $\Omega$ is the subset of the unit circle $\D$ given by:
\[
\Omega= \lac w \in \D \big{|} R< |w|< 1 \rac \quad \text{with} \quad R> \lb \frac{\sqrt{\beta}- 1}{\sqrt{\beta}+ 1} \rb,
\]
$\eta$ is a $\Cc^{2, \alpha}$-function on $\Omega \cup \S^1$ for some $\alpha \in (0, 1)$ and $h_{\beta}$ is the function:
\[
h_{\beta}(r)= \int_{|\log \beta|}^{2\log \lp \frac{1+ r}{1- r} \rp} \frac{\cosh t- \beta}{\sqrt{2\beta \cosh t- 1- \beta^2}} dt.
\]
Note that when $\eta \equiv 0$, the $1$-parameter family $(X_{\beta}^0)$ indexed by $\beta$ is the family of CMC-$1/2$ rotational examples --~see~\cite{SETo} for details~-- and if $\beta= 1$, $X_1^{\eta}$ is the end of an entire graph. We also have a simpler expression of $h_1$:
\[
h_1(r)= 2\frac{1+ r^2}{1- r^2}.
\]

\medskip

We compute the vertical flux $\sigma_3$ on a circle $\lac r= t \rac$ with $R< t< 1$. The infinitesimal generator of the vertical translations and the associated potential vector are respectively:
\[
S_3= E_3 \quad \text{and} \quad F_3= \frac{2r}{1+ r^2} (-\sin \theta E_1+ \cos \theta E_2),
\]
and making $t \into 1$, we obtain:
\[
\sigma_3= 2\pi \lp 1- \beta |e^{-\gamma}|_{L^2(\S^1)}^2 \rp \quad \text{with} \quad \gamma= \eta|_{r= 1}.
\]

If $\beta= 1$, the ends $X_1^{\eta}$ are ends of entire graphs, which implies in particular $\sigma_3= 0$ i.e. $|e^{-\gamma}|_{L^2(\S^1)}^2= 1$. It is precisely the necessary and sufficient condition of~\cite[Theorem~3.8]{CaHa}. And in the case of annuli, $\beta \neq 1$, the condition on the values at infinity in the definition of a $\beta$-deformable annulus at the beginning of~\cite[Subsection~5.2]{CaHa} is also the conservation of the vertical flux.

It is worth mentioning that in general the Noether invariants express only necessary conditions on the existence of a surface and not sufficient conditions as in the present case.

\medskip

Focus now on rotational annuli, the ends of which are parametrized by the immersions $X_{\beta}^0$ with $\beta \neq 1$. Similarly as for the computation of $\sigma_3$, we have $\sigma_1= \sigma_2= \sigma_R= 0$ and we already know $\sigma_3= 2\pi(1- \beta)$.

The expressions of the behaviors of Noether invariants under the action of isometries computed in Subsection~\ref{subsec:evolisome} show that the values of the flux and torque do not change when translating a rotational annulus. This point differs from the situation of minimal catenoids in $\R^3$, see for instance~\cite{Ro}. Indeed, in the space form $\R^3$, the torque has two more components and these components carry the information on the axis of the ends. It is no longer the case in $\H^2 \times \R$, which underlies the existence result~\cite[Theorem~5.10]{CaHa} of annuli without axis.

\vfill

\nidt Sébastien \textsc{Cartier}, Université Paris-Est, LAMA (UMR 8050), UPEMLV, UPEC, CNRS, F-94010, Créteil, France \\
e-mail: \verb+sebastien.cartier@u-pec.fr+


\begin{thebibliography}{1}

\bibitem{BeCa}
P.~Bérard and M.~P. Calvacante, \emph{Stability properties of rotational
  catenoids in the {H}eisenberg groups}, preprint arXiv:1010.0774, 2010.

\bibitem{BrGriGro}
R.~Bryant, P.~Griffiths, and D.~Grossman, \emph{Exterior differential systems
  and {E}uler-{L}agrange partial differential equations}, Chicago Lectures in
  Mathematics, University of Chicago Press, 2003.

\bibitem{CaHa}
S.~Cartier and L.~Hauswirth, \emph{Deformations of {CMC}-$1/ 2$ surfaces in
  $\mathbb{H}^2 \times \mathbb{R}$ with vertical ends at infinity}, preprint
  arXiv:1203.0760, 2012.

\bibitem{DaHa}
B.~Daniel and L.~Hauswirth, \emph{Half-space theorem, embedded minimal annuli
  and minimal graphs in the {H}eisenberg group}, Proc. Lond. Math. Soc. (3)
  \textbf{98} (2009), no.~2, 445--470.

\bibitem{SETo}
R.~Sa Earp and E.~Toubiana, \emph{Screw motion surfaces in $\mathbb{H}^2 \times
  \mathbb{R}$ and $\mathbb{S}^2 \times \mathbb{R}$}, Illinois J. Math.
  \textbf{49} (2005), no.~4, 1323--1362 (electronic).

\bibitem{No}
E.~Noether, \emph{Invariant variation problems}, Transport Theory Statist.
  Phys. \textbf{1} (1971), no.~3, 186--207, Translated from the German (Nachr.
  Akad. Wiss. G{\"o}ttingen Math.-Phys. Kl. II 1918, 235--257).

\bibitem{Ol}
P.~J. Olver, \emph{Applications of {L}ie groups to differential equations},
  second ed., Graduate Texts in Mathematics, vol. 107, Springer-Verlag, 1993.

\bibitem{Ro2}
P.~Romon, \emph{Noether theorem, conserved quantities, minimal and {CMC}
  surfaces}, in Oberwolfach Mini-Workshop \textit{Progress in surface theory},
  Report no. 21/2010.

\bibitem{Ro}
\bysame, \emph{Symmetries and conserved quantities for minimal surfaces},
  unpublished preprint, 1997.

\end{thebibliography}
\end{document}